\documentclass[letterpaper,11pt]{amsart}

\usepackage[margin=1.2in]{geometry}
\usepackage{amsmath,amsthm,amssymb}
\usepackage{xspace,xcolor}
\usepackage[breaklinks,colorlinks,citecolor=teal,linkcolor=teal,urlcolor=teal,pagebackref,hyperindex]{hyperref}
\usepackage[alphabetic]{amsrefs}
\usepackage[all]{xy}
\usepackage{color}
\usepackage{comment} 

\usepackage{tikz-cd}

\theoremstyle{plain}
\newtheorem{thm}{Theorem}[section]

\newtheorem{lem}[thm]{Lemma}

\newtheorem{cor}[thm]{Corollary}

\theoremstyle{definition}

\theoremstyle{remark}
\newtheorem{rmk}[thm]{Remark}

\newtheorem{claim}[thm]{Claim}

\def\Z{{\mathbf Z}}
\def\Q{{\mathbf Q}}
\def\R{{\mathbf R}}
\def\C{{\mathbf C}}

\def\A{{\mathbf A}}

\def\cD{\mathcal{D}}
\def\cE{\mathcal{E}}
\def\cF{\mathcal{F}}

\def\cH{\mathcal{H}}

\def\cM{\mathcal{M}}
\def\cN{\mathcal{N}}
\def\cO{\mathcal{O}}
\def\cP{\mathcal{P}}

\def\bff{{\bf f}}

\def\bD{{\bf D}}

\def\.{\cdot}
\def\^{\widehat}

\def\de{\partial}

\def\({\left(}
\def\){\right)}

\renewcommand{\and}{ \ \ \text{ and } \ \ }

\def\gr{{\mathrm{Gr}}}
\def\DR{\mathrm{DR}}

\DeclareMathOperator{\coker} {coker}

\DeclareMathOperator{\lct} {lct}

\makeatletter
\@namedef{subjclassname@2020}{\textup{2020} Mathematics Subject Classification}
\makeatother

\usepackage{soul}

\begin{document}

\author[Q.~Chen]{Qianyu Chen}

\address{Department of Mathematics, University of Michigan, 530 Church Street, Ann Arbor, MI 48109, USA}

\email{qyc@umich.edu}

\author[B.~Dirks]{Bradley Dirks}

\address{Department of Mathematics, University of Michigan, 530 Church Street, Ann Arbor, MI 48109, USA}

\email{bdirks@umich.edu}

\author[M.~Musta\c{t}\u{a}]{Mircea Musta\c{t}\u{a}}

\address{Department of Mathematics, University of Michigan, 530 Church Street, Ann Arbor, MI 48109, USA}

\email{mmustata@umich.edu}

\thanks{Q.C. was partially supported by NSF grant DMS-1952399 and M.M. was partially supported by NSF grants DMS-2001132 and DMS-1952399.}

\subjclass[2020]{14F10, 14B05, 14J17, 32S35}

\begin{abstract}
We show that if $Z$ is a local complete intersection subvariety of a smooth complex variety $X$, of pure codimension $r$, then
$Z$ has $k$-rational singularities if and only if $\widetilde{\alpha}(Z)>k+r$, where $\widetilde{\alpha}(Z)$ is the minimal exponent of $Z$. 
We also characterize this condition in terms of the Hodge filtration on the intersection complex Hodge module of $Z$. Furthermore,
we show that if $Z$ has $k$-rational singularities, then
the Hodge filtration on the local cohomology sheaf $\cH^r_Z(\cO_X)$ is generated
at level $\dim(X)-\lceil \widetilde{\alpha}(Z)\rceil-1$ and, assuming that $k\geq 1$ and $Z$ is singular, of dimension $d$, that $\cH^k(\underline{\Omega}_Z^{d-k})\neq 0$.
All these results have been known for hypersurfaces in smooth varieties.
\end{abstract}

\title[The minimal exponent and $k$-rationality] {The minimal exponent and $k$-rationality for local complete intersections}

\maketitle

\section{Introduction}

It is well-known that rational and Du Bois singularities play an important role in the hierarchy of singularities of higher-dimensional algebraic varieties. 
Recently, definitions of ``higher order" versions of these classes of singularities have been proposed, as follows. Suppose that $Z$ is a complex algebraic variety. 
If $\underline{\Omega}_Z^p$ is the  $p$-th graded piece of the Du Bois complex of $Z$ (suitably shifted), then there is a canonical morphism 
$$\Omega_Z^p\to\underline{\Omega}_Z^p$$
that is an isomorphism over the smooth locus of $Z$. Following \cite{Saito_et_al}, we say that $Z$ has $k$-Du Bois singularities if this morphism is an isomorphism
for $0\leq p\leq k$. For $k=0$, we recover the definition of Du Bois singularities.

On the other hand, if $\mu\colon\widetilde{Z}\to Z$ is a resolution of singularities that is an isomorphism over $Z\smallsetminus Z_{\rm sing}$
and such that $D=\mu^{-1}(Z_{\rm sing})$ is a simple normal crossing divisor on $\widetilde{Z}$, then following \cite{FL1} we say that $Z$ has
$k$-rational singularities if the canonical morphism
$$\Omega_Z^p\to {\mathbf R}\mu_*\Omega_{\widetilde{Z}}^p({\rm log}\,D)$$
is an isomorphism for $0\leq p\leq k$. Again, for $k=0$ this is the classical notion of rational singularities. 
Our main goal in this note is to characterize numerically, in the case when $Z$ is locally a complete intersection, the condition for having $k$-rational singularities.
A similar characterization for $k$-Du Bois local complete intersections has been obtained in \cite{MP1}, extending work on hypersurfaces in \cite{MOPW}
and \cite{Saito_et_al}.

Suppose that $X$ is a smooth, irreducible, $n$-dimensional complex algebraic variety and $Z$ is a local complete intersection closed subscheme of $X$,
of pure codimension $r$ in $X$. In this setting the \emph{minimal exponent} $\widetilde{\alpha}(Z)$ was introduced and studied in \cite{CDMO}. In the case
$r=1$, this is the invariant introduced by Saito in \cite{Saito_microlocal} as the negative of the largest root of the reduced Bernstein-Sato polynomial of $Z$.
In general, $\widetilde{\alpha}(Z)$ can be described in terms of the Kashiwara-Malgrange $V$-filtration associated to $Z$ and it is also related to the Hodge filtration on the 
local cohomology sheaf $\cH_Z^r(\cO_X)$. The minimal exponent can be considered as a refinement of the log canonical threshold of $(X,Z)$: we always have
${\rm lct}(X,Z)=\min\big\{\widetilde{\alpha}(Z),r\}$. Moreover, it is shown in \cite{CDMO} that $\widetilde{\alpha}(Z)>r$ if and only if $Z$ has rational singularities,
extending a result due to Saito \cite{Saito-B} in the case of hypersurfaces.

The following is our main result:

\begin{thm}\label{thm1_intro}
If $Z$ is a local complete intersection subvariety of the smooth, irreducible variety $X$, of pure codimension $r$, then $Z$ has $k$-rational singularities
if and only if $\widetilde{\alpha}(Z)>k+r$.
\end{thm}

In the case of hypersurfaces, this result was proved independently in \cite[Appendix]{FL2} and \cite{MP2}. The proof we give follows the idea in \cite{FL2},
making also essential use of results from \cite{CD} on the Kashiwara-Malgrange $V$-filtration in the case of higher codimension subvarieties. A key ingredient
in the proof is Saito's theory of mixed Hodge modules \cite{Saito_MHM}. 

The characterization of $k$-Du Bois singularities in \cite{MP1} for local complete intersections can also be formulated in terms of the minimal exponent:
it says that, with the notation in Theorem~\ref{thm1_intro}, $Z$ has $k$-Du Bois singularities if and only if $\widetilde{\alpha}(Z)\geq k+r$. In particular, we obtain the following

\begin{cor}\label{cor2_intro}
If $Z$ is a complex algebraic variety which is locally a complete intersection and if $Z$ has $k$-Du Bois singularities, for some $k\geq 1$,
then $Z$ has $(k-1)$-rational singularities. 
\end{cor}

Another consequence of the numerical characterizations of $k$-rational and $k$-Du Bois local complete intersections
is that $k$-rational implies $k$-Du Bois. However, this result has already been known (it was proved independently in \cite{FL2} and \cite{MP1})
and we use it in our proof of Theorem~\ref{thm1_intro}. 

As a consequence of the result in Theorem~\ref{thm1_intro} and of general properties 
of the minimal exponent, we obtain an upper bound for the dimension of the singular locus. 
We note that if in the following corollary we replace ``$k$-rational" by ``$k$-Du Bois", then it follows from the results in \cite{MP1}
that ${\rm codim}_Z(Z_{\rm sing})\geq 2k+1$.    

\begin{cor}\label{cor2.5_intro}
If $Z$ is a complex algebraic variety which is locally a complete intersection and if $Z$ has $k$-rational singularities, then 
$${\rm codim}_Z(Z_{\rm sing})\geq 2k+2.$$
\end{cor}

Let us recall the condition for $k$-Du Bois singularities in terms of the Hodge filtration on local cohomology. For every subvariety $Z$
of a smooth complex algebraic variety $X$ and every $i$, the local cohomology sheaf $\cH^i_Z(\cO_X)$ underlies a mixed Hodge module. As such, it carries a 
Hodge filtration $F_{\bullet}\cH^i_Z(\cO_X)$, an increasing filtration by coherent $\cO_X$-submodules. If $Z$ is a local complete intersection of pure codimension $r$,
then the only nonzero local cohomology sheaf is $\cH^r_Z(\cO_X)$. There is another filtration $E_{\bullet}\cH^r_Z(\cO_X)$ on $\cH^r_Z(\cO_X)$, 
also by coherent $\cO_X$-modules,
given by 
$$E_p\cH^r_Z(\cO_X)=\big\{u\in \cH^r_Z(\cO_X)\mid I_Z^{p+1}u=0\big\}\quad\text{for}\quad p\geq 0,$$
where $I_Z$ is the ideal defining $Z$. It is shown in \cite{MP1} that $F_p\cH^r_Z(\cO_X)\subseteq E_p\cH^r_Z(\cO_X)$ for all $p\geq 0$
and equality for $p=k$ implies equality also for $p<k$. 
One defines the \emph{cohomological level} of the Hodge filtration on $\cH^r_Z(\cO_X)$ by
$$p(Z)=\sup\big\{k\geq 0\mid F_k\cH^r_Z(\cO_X)=E_k\cH^r(\cO_X)\big\},$$
with the convention that $p(Z)=-1$ if there are no such $k$. It is then shown in \cite{MP1} that $Z$ has $k$-Du Bois singularities 
if and only if $p(Z)\geq k$. The condition in terms of the minimal exponent follows from this and the equality 
$p(Z)=\max\big\{\lfloor\widetilde{\alpha}(Z)\rfloor-r, -1\big\}$, proved in \cite{CDMO}.

We characterize $k$-rationality in a similar fashion. Recall that if $X$ is a smooth irreducible $n$-dimensional variety and $Z$ is a closed subvariety of $X$ of pure codimension $r$, then $\cH^r_Z(\cO_X)$ also carries a weight filtration and the lowest weight piece is $W_{n+r}\cH^r_Z(\cO_X)$, which underlies a pure
Hodge module of weight $n+r$ (this $\cD_X$-module is the \emph{intersection cohomology} $\cD_X$-module of Brylinski and Kashiwara \cite{BK}).
We prove the following result,
which in the case of hypersurfaces was proved in \cite{Olano}.

\begin{thm}\label{thm2_intro}
If $Z$ is a local complete intersection subvariety of the smooth, irreducible, $n$-dimensional variety $X$, of pure codimension $r$, then for every nonnegative integer 
$k$, we have $\widetilde{\alpha}(Z)>k+r$ if and only if
$F_kW_{n+r}\cH^r_Z(\cO_X)=E_k\cH^r_Z(\cO_X)$.
\end{thm}

We also show that for singular local complete intersections that have $k$-rational singularities, with $k\geq 1$, some higher cohomology groups of the graded 
pieces of the Du Bois complex do not vanish. This extends the result from \cite[Theorem~1.5]{MOPW} in the case of hypersurfaces.

\begin{thm}\label{thm_nonvanishing}
Let $Z$ be a local complete intersection subvariety of the smooth, irreducible, $n$-dimensional variety $X$. If $Z$ has pure dimension $d$ and $k$-rational singularities, for some $k\geq 1$, 
then 
\[
\cH^k\big(\underline\Omega^{d-k}_Z\big) \simeq \mathcal{E}{xt}^k_{\cO_Z}(\Omega_Z^k,\omega_Z) \simeq  \omega_Z \otimes_{\cO_Z} \mathrm{Sym}^k_{\mathcal O_Z} \mathcal Q,
\]
where $\mathcal Q$ is the cokernel of the canonical map $\mathcal T_X|_Z \to \mathcal N_{Z/X}$. In particular, if $Z$ is singular at $x$, then 
$\cH^k\big(\underline\Omega^{d-k}_Z\big)_x\neq 0$. 
\end{thm}

As observed in \cite{MOPW}, such a result imposes restrictions on varieties with quotient or toroidal singularities. 
Indeed, if $Z$ is a variety with quotient or toroidal singularities, then $\cH^i\big(\underline{\Omega}_Z^p)=0$ for all $p$ and all $i\geq 1$;
for quotient singularities, this follows from \cite[Section~5]{DuBois} and for toroidal singularities, it follows from \cite[Chapter~V.4]{GNPP}. 
On the other hand, it is well-known that such singularities are rational. By combining Theorems~\ref{thm1_intro} and \ref{thm_nonvanishing},
we thus obtain

\begin{cor}
Let $Z$ be a local complete intersection subvariety, of pure codimension $r$, of the smooth, irreducible algebraic variety $X$. If $Z$ is singular,
with quotient or toroidal singularities, then $r<\widetilde{\alpha}(Z)\leq r+1$. 
\end{cor}

Our final result concerns the level of generation of the Hodge filtration on 
$\cH^r_Z(\cO_X)$. Recall that if $\cM$ is a $\cD_X$-module endowed with a good filtration, where $\cD_X$ is the sheaf of differential operators on $X$, then we have
$F_1\cD_X\cdot F_p\cM\subseteq F_{p+1}\cM$, with equality for $p\gg 0$ (here $F_{\bullet}\cD_X$ is the order filtration on $\cD_X$). If equality
holds for $p\geq p_0$, we say that the filtration on $\cM$ is generated at level $p_0$. This definition applies, in particular, for the filtered $\cD_X$-module
underlying a mixed Hodge module on $X$.

\begin{thm}\label{thm4_intro}
If $Z$ is a singular, pure codimension $r$, local complete intersection subvariety of the smooth, irreducible, $n$-dimensional variety $X$, then the Hodge filtration 
on $\cH^r_Z(\cO_X)$ is generated at level $n-\lceil\widetilde{\alpha}(Z)\rceil-1$. 
\end{thm}

When $r=1$, this is \cite[Theorem~A]{MP0}. 
We also note that it follows from \cite[Theorem~4.2]{MP1} that the filtration on $\cH^r_Z(\cO_X)$ is always generated at level $n-r$, hence the assertion in the above theorem is interesting when $\widetilde{\alpha}(Z)>r-1$. Furthermore, via the equivalence in \emph{loc}.~\emph{cit}., the assertion in Theorem~\ref{thm4_intro}
admits the following interpretation in terms of relative vanishing.

\begin{thm}\label{thm5_intro}
Let $Z$ be a singular,  pure codimension $r$, local complete intersection subvariety of the smooth, irreducible, $n$-dimensional variety $X$. If $f\colon Y\to X$
is a proper morphism that is an isomorphism over $X\smallsetminus Z$, with $Y$ smooth and $E=f^{-1}(Z)_{\rm red}$ a simple normal crossing divisor, then
$$R^{r-1+i}f_*\Omega_Y^{n-i}({\rm log}\,E)=0\quad\text{for}\quad i>n-\lceil\widetilde{\alpha}(Z)\rceil-1.$$
\end{thm}

\medskip

\noindent {\bf Outline of the paper}. In the next section, we review some basic notions and results that we will need for the proofs of our main results.
Theorem~\ref{thm1_intro} and its corollaries, as well as Theorem~\ref{thm2_intro} are proved in Section~\ref{section3}. Theorem~\ref{thm_nonvanishing}
is proved in Section~\ref{section_nonvanishing}, 
while Theorem~\ref{thm4_intro} is proved in Section~\ref{section4}.

\noindent {\bf Acknowledgments}. We would like to thank Sebasti\'{a}n Olano and Mihnea Popa for many helpful discussions. 
We are also indebted to Christian Schnell for some useful suggestions and to Morihiko Saito for his comments on a previous version of this paper.

\section{Background overview}\label{section2}

In this section we recall some definitions and results that we will need.
We work over the field $\C$ of complex numbers. By a variety we mean a reduced scheme of finite type over $\C$, not necessarily irreducible. 
For a variety $Z$, we denote by $Z_{\rm sing}$ the singular locus of $Z$.

\subsection{Mixed Hodge modules}\label{section_MHM} We only give a brief introduction to mixed Hodge modules and refer for proofs and details to \cite{Saito_MHM}.
Let $X$ be a smooth, irreducible, $n$-dimensional variety and let $X^{\rm an}$ be the complex manifold corresponding to $X$. We denote by $\cD_X$ the sheaf of differential operators on $X$. For basic facts about $\cD_X$-modules,
we refer to \cite{HTT}. All the $\cD_X$-modules we will consider will be left $\cD_X$-modules. Since some of the results in the literature are stated for
right $\cD_X$-modules, we recall that there is an equivalence of categories between left and right $\cD_X$-modules  such that if $\cM^r$ is the right
$\cD_X$-module corresponding to the left $\cD_X$-module $\cM$, then we have an isomorphism of $\cO_X$-modules
$$\cM^r\simeq\cM\otimes_{\cO_X}\omega_X.$$ 
When dealing with filtered $\cD_X$-modules, the filtrations on $\cM$ and $\cM^r$ are indexed such that
the above isomorphism maps $F_{p-n}\cM^r$ to $F_p\cM\otimes_{\cO_X}\omega_X$ for all $p\in\Z$. 

All filtrations on $\cD_X$-modules that we will encounter are assumed to be bounded below, good filtrations compatible with the
filtration $F_{\bullet}\cD_X$ on $\cD_X$ by order of differential operators. This means that they are increasing, exhaustive filtrations by $\cO_X$-submodules such that we have
$$F_p\cD_X\cdot F_q\cM\subseteq F_{p+q}\cM\quad\text{for all}\quad p,q\in\Z,$$
and there is $q_0$ such that this inclusion is an equality for all $p\geq 0$ and $q\geq q_0$. In this case we say that the filtration is \emph{generated at level} $q_0$. 

A mixed Hodge module $M=(\cM, F_{\bullet}\cM, {\mathcal P}, \alpha, W_{\bullet}\cM)$ on $X$ consists of several pieces of data: $\cM$ is a $\cD_X$-module on $\cM$ (holonomic and with regular singularities), $F_{\bullet}\cM$ is a good filtration on $\cM$ (the \emph{Hodge filtration}), $W_{\bullet}\cM$ is a finite increasing 
filtration on $\cM$ by $\cD_X$-submodules (the \emph{weight filtration}), and $\cP$
is a perverse sheaf over $\Q$ on $X^{\rm an}$ (sometimes written as ${\rm rat}(M)$), whose complexification is isomorphic via $\alpha$ to the perverse sheaf over $\C$ that corresponds to $\cM$ via the
Riemann-Hilbert correspondence. These data are supposed to satisfy a complicated set of conditions that we do not discuss.  
We refer to $(\cM,F)$ as the filtered $\cD_X$-module underlying $M$ (though,
with an abuse of notation, we sometimes write $F_kM$ and $W_kM$ instead of $F_k\cM$ and $W_k\cM$, respectively).

The Tate twist $M(k)$ of a mixed Hodge module $M$ as above has the same underlying $\cD_X$-module, but the two filtrations are shifted by
$$F_i\cM(k)=F_{i-k}\cM\quad\text{and}\quad W_i\cM(k)=W_{i+2k}\cM\quad\text{for all}\quad i\in\Z.$$

We note that the mixed Hodge modules on $X$ form an Abelian category and every morphism of mixed Hodge modules is a morphism of $\cD_X$-modules, which 
preserves the Hodge and the weight filtration and is strict with respect to both filtrations.
There is a duality functor ${\mathbf D}$ on this category, lifting the usual duality functor on holonomic $\cD_X$-modules. All our Hodge
modules are polarizable, so the choice of a polarization implies that if $M$ as above is \emph{pure of weight} $k$ (that is, ${\rm Gr}^W_i(M)=0$ for $i\neq k$), we have an isomorphism
${\mathbf D}(M)\simeq M(k)$. For a general mixed Hodge module $M$ and for every $k\in\Z$, the graded piece ${\rm Gr}_k^W(M)$, with the induced Hodge filtration, is a pure
Hodge module of weight $k$. 

An important example of a mixed Hodge module (in fact, the only one that is easy to describe explicitly besides the ones with $0$-dimensional support) is $\Q_X^H[n]$, which is a pure Hodge module of weight $n$. 
The underlying $\cD_X$-module is $\cO_X$ and the Hodge filtration is such that ${\rm Gr}^F_i(\cO_X)=0$ for all $i\neq 0$. The corresponding perverse sheaf is $\Q_{X^{\rm an}}[n]$. 
Note that since $\Q_X^H[n]$ has weight $n$, a choice of polarization gives an isomorphism ${\mathbf D}(\Q_X^H[n]\big)\simeq\Q_X^H(n)[n]$. 

Given a mixed Hodge module $M$, with underlying filtered $\cD_X$-module $(\cM,F)$, the Hodge filtration makes the de Rham complex of $\cM$ a filtered complex. The graded pieces
are, in fact, complexes of $\cO_X$-modules. More precisely, ${\rm Gr}_p^F{\rm DR}_X(M)$ is the complex
$$0\to {\rm Gr}_p^F(\cM)\to\Omega_X^1\otimes_{\cO_X}{\rm Gr}_{p+1}^F(\cM)\to \ldots\to \Omega_X^n\otimes_{\cO_X}{\rm Gr}_{p+n}^F(\cM)\to 0,$$
placed in cohomological degrees $-n,\ldots,0$. For example, we have
$${\rm Gr}_{-p}^F{\rm DR}_X\big(\Q_X^H[n]\big)=\Omega_Y^p[n-p].$$
We always think of $\gr^F_p{\rm DR}_X(M)$ as an object in the derived category of coherent sheaves on $X$.
This construction is compatible with push-forward by proper morphisms (see \cite[Section~2.3.7]{Saito_MHP}) and satisfies the following compatibility property with the duality functor
by \cite[Sections 2.4.5 and 2.4.11]{Saito_MHP}: for every $p$, we have a canonical isomorphism
\begin{equation}\label{eq_compat_with_duality}
{\rm Gr}^F_p{\rm DR}_X\big({\mathbf D}(M)\big)\simeq \R\cH om_{\cO_X}\big({\rm Gr}^F_{-p}{\rm DR}_X(M),\omega_X[n]\big).
\end{equation}

For future reference, we include the following lemma, in which we consider arbitrary filtered $\cD_X$-modules:

\begin{lem}\label{general_lemma}
If $f\colon (\cM,F)\to (\cN,F)$
is a morphism of filtered $\cD_X$-modules on $X$ and $k\in\Z$, then the induced morphism
$${\rm Gr}^F_p{\rm DR}_X(f)\colon {\rm Gr}^F_p{\rm DR}_X(M)\to {\rm Gr}^F_p{\rm DR}_X(N)$$ is an isomorphism (in the derived category) for all $p\leq k$
if and only if $F_pf\colon F_pM\to F_pN$ is an isomorphism for all $p\leq k+n$.
\end{lem}

\begin{proof}
The ``if" assertion follows directly from the definition of the graded de Rham complex. For the converse, arguing by induction, it is enough to show that if $F_pf$ is an isomorphism
for all $p\leq k+n-1$ and $\gr^F_{k}{\rm DR}_X(f)$ is an isomorphism (in the derived category), then $F_{k+n}f$ is an isomorphism
(recall that all our filtrations are assumed to be bounded below).
By hypothesis, we have a morphism of complexes placed in cohomological degrees $-n,\ldots,0$:
\[
\begin{tikzcd}[column sep=tiny]
0\arrow{r} & \gr^F_k(M) \arrow{r}\arrow{d}{\cong} & \cdots \arrow[r, "\alpha"] & \Omega^{n-1}_X\otimes_{\cO_X}\gr_{k+n-1}^F(M) \arrow[r, "\beta"]\arrow{d}{\cong}   & \Omega^n_X\otimes_{\cO_X} \gr^F_{k+n}(M)\arrow{r}\arrow{d} & 0 \\
0\arrow{r} & \gr^F_k(N) \arrow{r} &  \cdots \arrow[r, "\gamma"] &  \Omega^{n-1}_X\otimes_{\cO_X} \gr^F_{k+n-1}(N) \arrow[r, "\delta"] & \Omega^n_X\otimes_{\cO_X} \gr^F_{k+n}(N) \arrow{r} & 0
\end{tikzcd}
\]
such that the vertical maps in degrees $\neq 0$ are isomorphisms and such that all induced maps in cohomology are isomorphisms.
The first condition implies that the map $\coker(\alpha)\to \coker(\gamma)$ is an isomorphism and since 
the map induced for the $(-1)$-cohomology is an isomorphism, it follows from the 5-Lemma that the map ${\rm Im}(\beta)\to {\rm Im}(\delta)$ is an isomorphism. 
Since the map induced for the $0$-cohomology is an isomorphism, it follows from the $5$-Lemma 
that the map ${\rm Gr}^F_{k+n}(f)$ is an isomorphism, and one more application of the 5-Lemma
implies that $F_{k+n}f$ is an isomorphism. 
\end{proof}

One can define mixed Hodge modules also on a singular variety $Z$. In our setting, $Z$ will be embedded in a fixed smooth variety $X$,
and we will \emph{always} view the mixed Hodge modules on $Z$ as mixed Hodge modules on $X$
whose support is contained in $Z$. One can consider the \emph{bounded derived category of mixed Hodge modules} on $Z$, denoted $D^b\big({\rm MHM}(Z)\big)$. 
One can show that this is equivalent to the subcategory of $D^b\big({\rm MHM}(X)\big)$ consisting of objects whose cohomology is supported on $Z$ (see \cite[Corollary~2.23]{Saito_MHM}).
We will denote by $\cH^p$ the standard $p$-th cohomology functor $D^b\big({\rm MHM}(Z)\big)\to {\rm MHM}(Z)$. 
The derived category of mixed Hodge modules satisfies a 6-functor formalism. For example, if $i\colon Z\hookrightarrow X$ is the inclusion, where $X$ is smooth, then the underlying $\cD_X$-module of the mixed Hodge module $\cH^p\big(i^!\Q_X^H[n]\big)$
is the local cohomology sheaf $\cH^p_Z(\cO_X)$ of $\cO_X$ along $Z$. With a slight abuse of notation, from now one we will denote by $\cH^p_Z(\cO_X)$
also the corresponding mixed Hodge module.
For every variety $Z$, if $a_Z\colon Z\to {\rm pt}$ is the morphism to a point, then one defines $\Q_Z^H:=a_Z^*(\Q_{\rm pt}^H)$ in $D^b\big({\rm MHM}(Z)\big)$. If $Z$ is smooth, then this coincides
(up to a cohomological shift) with the object that we have already discussed. In general, however, it is a more complicated object.
If $X$ is a smooth, irreducible $n$-dimensional variety and $i\colon Z\hookrightarrow X$ is a closed embedding, then 
by functoriality we have
a canonical isomorphism $\Q_Z^H\simeq i^*\Q_X^H$, so we have a canonical isomorphism
\begin{equation}\label{eq1_dual_Q}
{\mathbf D}(\Q_Z^H)\simeq i^!\Q_X^H(n)[2n].
\end{equation}

For every $Z$, it is shown in 
\cite[Section~4.5]{Saito_MHM} that $\Q_Z^H$ is of weight $\leq 0$, that is, we have ${\rm Gr}^W_i\big(\cH^j(\Q_Z^H)\big)=0$ for $i>j$. Furthermore, if $Z$ has pure dimension $d$,
then $\cH^i(\Q_Z^H)=0$ for $i>d$ and the \emph{intersection complex Hodge module}
\begin{equation}\label{definition_IC}
{\rm IC}_Z\Q^H:={\rm Gr}^W_d\cH^d(\Q_Z^H)
\end{equation}
can be characterized as the unique object of ${\rm MHM}(Z)$ whose restriction to $U=Z\smallsetminus Z_{\rm sing}$ is $\Q_U^H[d]$ and which has no subobject or quotient supported on
$Z_{\rm sing}$. The corresponding perverse sheaf is the intersection complex of $Z$; if $Z$ is irreducible, then
 this is simple, hence so is ${\rm IC}_Z\Q^H$ and we have $\Q={\rm End}\big({\rm IC}_Z\Q^H\big)$. In general, if $Z$ has $N$ irreducible
 components, we have ${\rm End}\big({\rm IC}_Z\Q^H\big)=\Q^N$ and a morphism $\big({\rm IC}_Z\Q^H\big)\to \big({\rm IC}_Z\Q^H\big)$
 is uniquely determined by its restriction to the smooth locus of $Z$.

Note that by definition of ${\rm IC}_Z\Q^H$, we have a canonical morphism 
\begin{equation}\label{def_gamma}
\gamma_Z\colon\Q_Z^H[d]\to {\rm IC}_Z\Q^H.
\end{equation}

Suppose now that $X$ is a smooth, irreducible $n$-dimensional variety and $i\colon Z\hookrightarrow X$ is a closed embedding. Let $r=n-d$. 
Since ${\rm IC}_Z\Q^H={\rm Gr}^W_d\cH^d(\Q_Z^H)$ and ${\rm Gr}^W_p\cH^d(\Q_Z^H)=0$ for $p>d$, it follows using (\ref{eq1_dual_Q}) that
\begin{equation}\label{eq_dual_IC}
{\mathbf D}({\rm IC}_Z\Q^H)\simeq{\rm Gr}^W_{-d}\big(\cH^{-d}{\mathbf D}(\Q_Z^H)\big)\simeq{\rm Gr}^W_{-d}\cH^{-d}\big(i^!\Q_X^H(n)[2n]\big)={\rm Gr}^W_{n+r}\cH^r_Z(\cO_X)(n)
\end{equation}
and ${\rm Gr}^W_p\cH^r_Z(\cO_X)=0$ for $p<n+r$. We note that this lowest
weight piece of $\cH^r_Z(\cO_X)$ is the \emph{intersection cohomology $\cD$-module} introduced by Brylinski and Kashiwara in \cite{BK};
if $Z$ is irreducible, then it can be characterized as the unique simple $\cD_X$-submodule of $\cH^r_Z(\cO_X)$.

We also consider the shifted dual $\gamma_Z^{\vee}={\mathbf D}(\gamma)(-d)$ of $\gamma_Z$, that can be identified via (\ref{eq1_dual_Q})
to 
\begin{equation}\label{def_gamma_2}
\gamma_Z^{\vee}\colon {\mathbf D}({\rm IC}_Z\Q^H)(-d)\to i^!\Q_X^H[n+r](n-d).
\end{equation}
Note that since ${\rm IC}_Z\Q^H$ is pure of weight $d$, the choice of a polarization gives an isomorphism
${\mathbf D}({\rm IC}_Z\Q^H)(-d)\simeq {\rm IC}_Z\Q^H$. 

We will be especially interested in the case when $Z$ is a local complete intersection subvariety of $X$, of pure codimension $r$.
In this case $\cH^i_Z(\cO_X)=0$ for all $i\neq r$, hence $i^!\Q_X^H[n+r]$ is a mixed Hodge module on $Z$. Duality implies that
also $\Q_Z^H[d]$ is a mixed Hodge module on $Z$, hence $\gamma_Z$ and $\gamma_Z^{\vee}$ are morphisms of mixed Hodge
modules.

\subsection{$V$-filtrations}\label{section_V_filtration} Suppose that $X$ is a smooth, irreducible, $n$-dimensional affine variety and $f_1,\ldots,f_r\in\cO_X(X)=R$ are nonzero regular functions such that the ideal
$(f_1,\ldots,f_r)$ defines the closed subscheme $Z$ of $X$. We consider the graph embedding 
$$\iota\colon X\hookrightarrow W=X\times\A^r,\quad \iota(x)=\big(x,f_1(x),\ldots,f_r(x)\big)$$
and the $\cD$-module pushforward $B_{\bff}=\iota_+\cO_X$ (where $\bff$ stands for $(f_1,\ldots,f_r)$). If $t_1,\ldots,t_r$ denote the standard coordinates on 
$\A^r$, then we can write 
$$B_{\bff}=\bigoplus_{\alpha\in \Z_{\geq 0}^r}R\partial_t^{\alpha}\delta_{\bff},$$
where for $\alpha=(\alpha_1,\ldots,\alpha_r)$, we put $\partial_t^{\alpha}=\partial_{t_1}^{\alpha_1}\cdots \partial_{t_r}^{\alpha_r}$. 
The action of $R$ and of $\partial_{t_i}$ are the obvious ones, while the actions of $D\in {\rm Der}_{\C}(R)$ and of the $t_i$ are given by 
$$D\cdot h\partial_t^{\alpha}\delta_{\bff}=D(h)\partial_t^{\alpha}\delta_{\bff}-\sum_{i=1}^rD(f_i)h\partial_t^{\alpha+e_i}\delta_{\bff}\quad\text{and}\quad
t_i\cdot h\partial_t^{\alpha}\delta_{\bff}=f_ih\partial_t^{\alpha}\delta_{\bff}-\alpha_ih\partial_t^{\alpha-e_i}\delta_{\bff},$$
where $e_1,\ldots,e_r$ is the standard basis of $\Z^d$. In fact, $B_{\bff}$ underlies the pure Hodge module $\iota_*\Q_X^H[n]$, of weight $n$,
with the Hodge filtration given by 
$$F_{p+r}B_{\bff}=\bigoplus_{|\alpha|\leq p}R\partial_t^{\alpha}\delta_{\bff},$$
where for $\alpha=(\alpha_1,\ldots,\alpha_r)$, we put $|\alpha|=\alpha_1+\ldots+\alpha_r$. 

The $V$-filtration on $B_{\bff}$ has been constructed by Kashiwara \cite{Kashiwara}, extending work of Malgrange \cite{Malgrange} in the case $r=1$.
(Actually, in both of these references, the $V$-filtration is indexed by integers. The $\Q$-indexed version that we discuss below was introduced by Saito \cite{Saito_GM}.)
It is a decreasing, exhaustive filtration indexed by rational numbers $(V^{\lambda}B_{\bff})_{\lambda\in\Q}$. It is discrete and left-continuous and 
it is characterized by several properties, the most important of these saying that for every $\lambda\in\Q$
$$t_i\cdot V^{\lambda}B_{\bff}\subseteq V^{\lambda+1}B_{\bff}\quad\text{and}\quad \partial_{t_i}\cdot V^{\lambda}B_{\bff}\subseteq V^{\lambda-1}B_{\bff},$$
and if $s=-\sum_{i=1}^r\partial_{t_i}t_i$, then $s+\lambda$ is nilpotent on ${\rm Gr}^{\lambda}_V(B_{\bff})=V^{\lambda}B_{\bff}/V^{>\lambda}B_{\bff}$,
where $V^{>\lambda}B_{\bff}=\bigcup_{\beta>\lambda}V^{\beta}B_{\bff}$. Note that the Hodge filtration on $B_{\bff}$ induces a Hodge filtration
on each ${\rm Gr}_V^{\lambda}(B_{\bff})$. 

In fact, a $V$-filtration exists on $\iota_+\cM$, whenever $\cM$ underlies 
a mixed Hodge module. In the case $r=1$, the interplay between the Hodge filtration on $\cM$ and $V$-filtrations plays an important role in the definition of mixed Hodge modules. For details about the construction and properties of $V$-filtrations, see \cite{BMS}. 

Let $i\colon Z\hookrightarrow X$ be the inclusion. For $r=1$, the $V$-filtration is the key ingredient for the definition of $i^!(M)$ and $i^*(M)$ when $M$ is a 
mixed Hodge module on $X$. In the case $r>1$, the corresponding description does not follow from the definition of these functors, but it has been recently 
proved in \cite[Theorem~1.2]{CD}. We only state this in the case $M=\Q_X^H[n]$. 

\begin{thm}\label{thm_CD}
With the above notation, the following hold: the Koszul-type complex 
$$0\to \mathrm{Gr}^0_V(B_{\mathbf f})(-r) \xrightarrow{(t_1,t_2,\dots,t_r)} \bigoplus^r_{i=1}\mathrm{Gr}^1_V(B_{\mathbf f})(-r) \to \cdots \to 
\mathrm{Gr}^r_V(B_{\mathbf f})(-r)\to 0,$$
placed in cohomological degrees $0,\ldots,r$ represents $i^!\Q_X^H[n]$ in the derived category of filtered $\cD_X$-modules
and the Koszul-type complex
$$0\to {\rm Gr}_V^r(B_{\bff})\xrightarrow{(\partial_{t_1},\dots,\partial_{t_r})} \bigoplus_{i=1}^r{\rm Gr}_V^{r-1}(B_{\bff})(-1)\to\cdots\to {\rm Gr}_V^0(B_{\bff})(-r)\to 0$$
placed in cohomological degrees $-r,\ldots,0$ represents $i^*\Q_X^H[n]$. 
\end{thm}

\subsection{The minimal exponent}\label{section_min_exponent} 
We next discuss the minimal exponent for local complete intersection varieties, following \cite{CDMO}. Let $X$ be a smooth,
irreducible, $n$-dimensional variety and $Z$ a (nonempty) closed subscheme of $X$, which is locally a complete intersection of pure codimension $r$. Suppose first 
that $X={\rm Spec}(R)$ is affine
and $Z$ is defined by the ideal generated by $f_1,\ldots,f_r\in R$. The minimal exponent $\widetilde{\alpha}(Z)$ is defined by\footnote{We note that what we
denote by $F_{p+r}B_{\bff}$ here is denoted by $F_pB_{\bff}$  in \cite{CDMO}.}
\begin{equation}\label{eq3_intro_v2}
\widetilde{\alpha}(Z)
= \left\{
\begin{array}{cl}
\sup\{\gamma>0\mid \delta_{\mathbf f}\in V^{\gamma}B_{\bff}\} , & \text{if}\,\,\delta_{\mathbf f}\not\in V^rB_{\mathbf f}; \\[2mm]
\sup\{r-1+q+\gamma\mid F_{q+r}B_{\bff}\subseteq V^{r-1+\gamma}B_{\bff}\} , & \text{if}\,\,\delta_{\mathbf f}\in V^rB_{\mathbf f}.
\end{array}\right.
\end{equation}
In general, we consider a cover $X=U_1\cup\ldots\cup U_N$, where each $U_i$ is an affine open subset as above, and put
$$\widetilde{\alpha}(Z)=\min_{i;Z\cap U_i\neq\emptyset}\widetilde{\alpha}(Z\cap U_i).$$

It follows from \cite[Theorem~1]{BMS} that we always have $\min\big\{\widetilde{\alpha}(Z),r\big\}=\lct(X,Z)$, the \emph{log canonical threshold} of the pair $(X,Z)$. Therefore the minimal exponent is interesting
precisely when $\lct(X,Z)=r$, in which case $Z$ is automatically reduced (see \cite[Remark~4.2]{CDMO}). Moreover, it follows from \cite[Corollary~1.7]{CDMO}
that $Z$ has rational singularities if and only if $\widetilde{\alpha}(Z)>r$. One can also show (see \cite[Remark~4.15]{CDMO}) that $Z$ is smooth if and only if $\widetilde{\alpha}(Z)=\infty$; by definition
of the minimal exponent, this can be rephrased as
\begin{equation}\label{eq_char_V_smooth}
Z\,\,\text{is smooth}\quad\text{if and only if}\quad V^rB_{\bff}=B_{\bff}.
\end{equation}
In fact,
if $x\in Z$ is a singular point, then we have the following more precise bound (see \cite[Remark~4.21]{CDMO}):
\begin{equation}\label{eq_bd_min_exp}
\widetilde{\alpha}(Z)\leq n-\tfrac{1}{2}\dim_{\C}T_xZ.
\end{equation}
The minimal exponent $\widetilde{\alpha}(Z)$ depends on the ambient variety $X$, but in a predictable way: the difference $\widetilde{\alpha}(Z)-\dim(X)$
only depends on $Z$ (see \cite[Proposition~4.14]{CDMO}). 

When $r=1$, the minimal exponent was defined by Saito \cite{Saito_microlocal} as the negative of the largest root of the \emph{reduced Bernstein-Sato polynomial}
$\widetilde{b}_Z(s)$. For the fact that this agrees with the above definition, see for example \cite[Lemma~5.3 and Corollary~C]{MP5}.

Recall now that the $\cD_X$-module $\cH^r_Z(\cO_X)$ underlies a mixed Hodge module on $X$, namely $\cH^r\big(i^!\Q_X^H[n]\big)$, where
$i\colon Z\hookrightarrow X$ is the inclusion. We thus have a canonical filtration on $\cH^r_Z(\cO_X)$, the Hodge filtration $\big(F_p\cH^r_Z(\cO_X)\big)_{p\geq 0}$.
We have a second filtration, the \emph{order} filtration $\big(E_p\cH^r_Z(\cO_X)\big)_{p\geq 0}$, given by
$$E_p\cH^r_Z(\cO_X)=\big\{u\in\cH^r_Z(\cO_X)\mid I_Z^{p+1}u=0\big\}={\rm Im}\big({\mathcal Ext}^r_{\cO_X}(\cO_X/I_Z^{p+1},\cO_X)\hookrightarrow
\cH^r_Z(\cO_X)\big),$$
where $I_Z$ is the ideal defining $Z$ in $X$ (see \cite[Proposition~3.11]{MP1}). 
 It is a general fact that $F_p\cH^r_Z(\cO_X)\subseteq E_p\cH^r_Z(\cO_X)$ for all $p\geq 0$
(see \cite[Proposition~3.4]{MP1}) and the following result shows that the minimal exponent governs how far these two filtrations agree
(see \cite[Theorem~1.3]{CDMO}):

\begin{thm}\label{thm_local_coho}
If $X$ is a smooth, irreducible variety and $Z$ is a local complete intersection subvariety of pure codimension $r$ in $X$, then for a nonnegative integer $k$, we have
$F_p\cH^r_Z(\cO_X)=E_p\cH^r_Z(\cO_X)$ for $0\leq p\leq k$ if and only if $\widetilde{\alpha}(Z)\geq r+k$. 
\end{thm}

\subsection{$k$-Du Bois singularities}\label{section_k_DuBois}
To a variety $Z$, Du Bois associated in \cite{DuBois} a complex $\underline{\Omega}_Z^{\bullet}$, known now as the 
\emph{Du Bois complex} of $Z$. This is a filtered complex that agrees with the de Rham complex $\Omega_Z^{\bullet}$, with the ``stupid" filtration, when $Z$ is smooth. 
This allows extending to singular varieties some important cohomological properties of the de Rham complex of smooth varieties, see \cite[Chapter~7.3]{PetersSteenbrink}
for an introduction to this topic. 

We are interested in the shifted truncations $\underline{\Omega}_Z^p:={\rm Gr}_F^p(\underline{\Omega}_Z^{\bullet})[p]$, which are objects in the bounded derived category
$D^b_{\rm coh}(Z)$ of coherent sheaves on $Z$. 
For every $p$, there is a canonical morphism $\Omega_Z^p\to \underline{\Omega}_Z^p$ that is an isomorphism over the smooth locus of $Z$. Following 
\cite{Saito_et_al}, we say that $Z$ has $k$-Du Bois singularities\footnote{Strictly speaking, one should say ``has at most $k$-Du Bois singularities", since we do not require $Z$ to be singular. However,
we trust that the simplified formulation will not lead to confusion.}, for a nonnegative integer $k$, if these morphisms are isomorphisms for all $0\leq p\leq k$. 
Note that for $k=0$, we recover the familiar notion of \emph{Du Bois singularities}.

As we have mentioned in the Introduction, it was shown in \cite[Theorem~F]{MP1} that if $X$ is a smooth, irreducible variety and $Z$ is a local complete intersection
subvariety of $X$, of pure codimension $r$, then $Z$ has $k$-Du Bois singularities if and only if $F_p\cH^r_Z(\cO_X)=E_p\cH^r_Z(\cO_X)$ for $p\leq k$. In terms of minimal exponents,
this condition can be rephrased as $\widetilde{\alpha}(Z)\geq r+k$. The proof of this result in \emph{loc}. \emph{cit}. extends the argument in the case of hypersurfaces, 
for which the two implications had previously been proved in 
\cite{MOPW} and \cite{Saito_et_al}.

\begin{rmk}\label{sing_locus_Du_Bois}
If $Z$ is a local complete intersection variety with $k$-Du Bois singularities, then
${\rm codim}_Z(Z_{\rm sing})\geq 2k+1$. Indeed, this is a local statement, hence we may assume that $Z$ has pure dimension (we use the fact that 
$Z$ is Cohen-Macaulay) and that it is a closed subvariety of the smooth irreducible variety $X$. In this case the assertion follows by combining
\cite[Corollary~3.40 and Theorem~F]{MP1}. 
\end{rmk}

The connection between the Du Bois complex and mixed Hodge modules is provided by the following result of Saito. If $Z$ is a closed subvariety
of the smooth, irreducible, $n$-dimensional variety $X$ and $i\colon Z\hookrightarrow X$ is the inclusion, then it is a consequence of 
\cite[Theorem~4.2]{Saito-HC} that for every $p$, we have an isomorphism
\begin{equation}\label{eq_Saito_DuBois}
\underline{\Omega}_Z^p[-p]\simeq {\rm Gr}^F_{-p}{\rm DR}_X(\Q_Z^H)
\end{equation}
in $D^b_{\rm coh}(X)$.
In light of (\ref{eq_compat_with_duality}) and (\ref{eq1_dual_Q}), this is equivalent to
\begin{equation}\label{eq2_Saito_DuBois}
\underline{\Omega}_Z^p[-p]\simeq \R\cH om_{\cO_X}\big({\rm Gr}^F_{p-n}{\rm DR}_Xi^!\Q_X^H[n],\omega_X\big).
\end{equation}
For an easy proof of this isomorphism, see \cite[Proposition~5.5]{MP1}.

\subsection{$k$-rational singularities}\label{section_k_rational}
Given a variety $Z$, by a \emph{strong log resolution} of $Z$ we mean a proper morphism 
$\mu\colon\widetilde{Z}\to Z$ that is an isomorphism over $Z\smallsetminus Z_{\rm sing}$, such that $\widetilde{Z}$ is smooth and
$E=\mu^{-1}(Z_{\rm sing})$ is a simple normal crossing divisor. For a nonnegative integer $k$, following \cite{FL1}, we say that $Z$ has 
\emph{$k$-rational} singularities if the canonical morphism
\begin{equation}\label{eq_def_k_rat}
\Omega_Z^p\to \R\mu_*\Omega^p_{\widetilde{Z}}({\rm log}\,E)
\end{equation}
is an isomorphism for all $p\leq k$. This is easily seen to be independent of the log resolution (see for example \cite[Lemma~1.6]{MP2}). 
Note that for $k=0$ we recover the classical notion of rational singularities. This condition implies that $Z$ is normal, hence in particular,
every connected component of $Z$ is irreducible. 
The notion of $k$-rational singularities has been extensively studied in \cite{FL1}, \cite{FL2}, \cite{FL3}, \cite{MP2}. 

For our purpose it will be convenient to consider a different description of $k$-rational singularities. Recall from
\cite[Section~6]{MP2} that for every  variety $Z$ of pure dimension $d$ and every nonnegative integer $k$, we have a canonical morphism 
\begin{equation}\label{eq_can1}
\psi_k\colon \underline{\Omega}_Z^k\to \R\cH om_{\cO_Z}\big(\underline{\Omega}_Z^{d-k},\omega_Z^{\bullet}[-d]\big),
\end{equation}
where $\omega_Z^{\bullet}$ is the dualizing complex of $Z$. This is defined as follows: suppose that $\mu\colon Y\to Z$ 
is an arbitrary resolution of singularities (we only require that $Y$ is smooth and $\mu$ is proper and an isomorphism
over a dense open subset of $Z$). By functoriality of the Du Bois complex, for every nonnegative integer $k$, we have a canonical morphism 
$\alpha_k\colon\underline{\Omega}_Z^k\to \R\mu_*\Omega_Y^k$. On the other hand, on $Y$ we have a canonical isomorphism
$$\Omega_Y^k\overset{\simeq}\longrightarrow \R\cH om_{\cO_Y}\big(\Omega_Y^{d-k},\omega_Y^{\bullet}[-d]\big).$$
By pushing this forward and using Grothendieck duality for $\mu$, we obtain an isomorphism $\beta_k$ as the composition
$$\R\mu_*\Omega_Y^k\overset{\simeq}\longrightarrow \R\mu_*\R\cH om_{\cO_Y}\big(\Omega_Y^{d-k},\omega_Y^{\bullet}[-d]\big)\overset{\simeq}\longrightarrow
\R\cH om_{\cO_Z}\big(\R\mu_*\Omega_Y^{d-k},\omega_Z^{\bullet}[-d]\big).$$
The morphism $\psi_k$ is obtained as the composition $\alpha_{d-k}^{\vee}\circ \beta_k\circ\alpha_k$, where we put
$\alpha_{d-k}^{\vee}=\R\cH om_{\cO_Z}\big(\alpha_{d-k},\omega_Z^{\bullet}[-d]\big)$. 
It is shown in \cite[Proposition~6.1]{MP2} that this definition does not depend on the choice of resolution of singularities.

With this notation, we have the following characterization of $k$-rational singularities in the local complete intersection case,
see \cite{MP2} (in \emph{loc}. \emph{cit}. one assumes that $Z$ is irreducible, but the argument works in general):

\begin{thm}\label{chr_k_rat}
If $Z$ is a local complete intersection variety of pure dimension $d$ and $k$ is a nonnegative integer, then $Z$ has $k$-rational singularities
if and only if $Z$ has $k$-Du Bois singularities and the morphism $\psi_k\colon \underline{\Omega}_Z^k\to \R\cH om_{\cO_Z}\big(\underline{\Omega}_Z^{d-k},\omega_Z\big)$ is an isomorphism. 
\end{thm}

It will be important for us to use an interpretation of the morphism $\psi_k$ from \cite[Appendix]{FL2}, as the graded de Rham of a
morphism of mixed Hodge modules. Let $Z$ be a variety of pure dimension $d$ and $\mu\colon Y\to Z$ any resolution of singularities,
with $\mu$ a projective morphism.
Note that by functoriality we have a canonical morphism of mixed Hodge modules $\alpha\colon \Q_Z^H[d]\to\mu_*\Q^H_Y[d]$. On the other hand, 
since $\Q_Y^H[d]$ is pure of weight $d$, 
on $Y$ we have a canonical
isomorphism $\Q_Y^H[d]\to {\mathbf D}\big(\Q_Y^H[d]\big)(-d)$, which after pushing forward to $Z$ and using the compatibility of pushforward with duality,
gives an isomorphism
$$\beta\colon \mu_*\Q_Y^H[d]\overset{\simeq}\longrightarrow \mu_*{\mathbf D}\big(\Q_Y^H[d]\big)(-d)\overset{\simeq}\longrightarrow
{\mathbf D}\big(\mu_*\Q_Y^H[d]\big)(-d).$$
We then obtain a morphism $\psi_Z$ in the derived category of mixed Hodge modules on $Z$ as the following composition
\begin{equation}\label{eq_psi}
\Q_Z^H[d]\overset{\alpha}\longrightarrow \mu_*\Q_Y^H[d]\overset{\beta}\longrightarrow {\mathbf D}\big(\mu_*\Q_Y^H[d]\big)(-d)
\overset{\alpha^{\vee}}\longrightarrow
{\mathbf D}\big(\Q_Z^H[d]\big)(-d),
\end{equation}
where $\alpha^{\vee}={\mathbf D}(\alpha)(-d)$.

If $Z$ is a closed subvariety of the smooth, irreducible variety $X$ and $i\colon Z\hookrightarrow X$ is the inclusion, then using the compatibility of 
the graded de Rham complex with direct image and duality, we see that for every $k\in\Z$ we have
$$\alpha_k={\rm Gr}^F_{-k}{\rm DR}_X(\alpha)[k-d], \,\,\beta_k={\rm Gr}^F_{-k}{\rm DR}_X(\beta)[k-d],\,\,\alpha_{d-k}^{\vee}=
{\rm Gr}^F_{-k}{\rm DR}_X(\alpha^{\vee})[k-d],$$
hence $\psi_k={\rm Gr}^F_{-k}{\rm DR}_X(\psi_Z)[k-d]$. 

\begin{rmk}\label{dual_psi}
It follows from the definition of $\psi_Z$ that $\psi_Z^{\vee}:={\mathbf D}(\psi_Z)(-d)$ can be identified with $\psi_Z$. 
\end{rmk}

\begin{rmk}\label{rmk_two_morphisms}
Suppose now that $X$ is a smooth, irreducible, $n$-dimensional variety and $i\colon Z\hookrightarrow X$ is a closed embedding,
where $Z$ is a local complete intersection subvariety of $X$, of pure codimension $r$. Let $d=n-r$. 
As we have already mentioned, in this case, the morphisms
$$\Q^H_Z[d]\overset{\gamma_Z}\longrightarrow \mathrm{IC}_Z\Q^H\quad\text{and}\quad {\mathbf D}\big(\mathrm{IC}_Z\Q^H\big)(-d)
\overset{\gamma_Z^{\vee}}\longrightarrow {\mathbf D}\big(\Q^H_Z[d]\big)(-d)$$
are morphisms of mixed Hodge modules, with $\gamma_Z$ surjective and $\gamma_Z^{\vee}$ injective. 

Since $\psi_Z$ is a morphism between two mixed Hodge modules on $Z$, we obtain the same morphism
if we take $\cH^0(-)$; in other words, $\psi_Z$ agrees with the composition
\begin{equation}\label{eq_fact1}
\Q_Z^H[d]\to \cH^0\big(\mu_*\Q_Y^H[d]\big)\to {\mathbf D}\big(\cH^0(\mu_*\Q_Y^H[d])\big)(-d)\to {\mathbf D}\big(\Q_Z^H[d]\big)(-d).
\end{equation}
On the other hand, since $\Q^H_Y[d]$ is pure of weight $d$, so is $\cH^0\big(\mu_*\Q_Y^H[d]\big)$, see \cite[(4.5.2)]{Saito_MHM}. 
Since $\gr_W^i\big(\Q_Z^H[d]\big)=0$ for $i>d$, it follows that
 the composition in (\ref{eq_fact1}) further factors as
\begin{equation}\label{eq_fact2}
\Q_Z^H[d]\overset{\gamma_Z}\longrightarrow \mathrm{IC}_Z\Q^H\to \cH^0\big(\mu_*\Q_Y^H[d]\big)\to {\mathbf D}\big(\cH^0(\mu_*\Q_Y^H[d])\big)(-d)
\end{equation}
$$\to 
{\mathbf D}\big(\mathrm{IC}_Z\Q^H)(-d)\overset{\gamma_Z^{\vee}}\longrightarrow {\mathbf D}\big(\Q_Z^H[d]\big)(-d).$$
We also note that the intermediate composition
$$\mathrm{IC}_Z\Q^H\to \cH^0\big(\mu_*\Q_Y^H[d]\big)\to {\mathbf D}\big(\cH^0(\mu_*\Q_Y^H[d])\big)(-d)\to 
{\mathbf D}\big(\mathrm{IC}_Z\Q^H)(-d)$$
is always an isomorphism. Indeed, a morphism $\mathrm{IC}_Z\Q^H\to {\mathbf D}\big(\mathrm{IC}_Z\Q^H)(-d)$ is uniquely determined
by its restriction to a dense open subset of the smooth locus of $Z$, and on a suitable such subset over which $\mu$ is an isomorphism
this composition is the identity.

For every $k$, it follows from Lemma~\ref{general_lemma} that ${\rm Gr}_p^F{\rm DR}_X(\psi_Z)$ is an isomorphism for all $p\leq k$
if and only if $F_p\psi_Z$ is an isomorphism for every $p\leq k+n$. Since $\gamma_Z$ is surjective and $\gamma_Z^{\vee}$ is injective,
it follows from the above discussion that ${\rm Gr}_p^F{\rm DR}_X(\psi_Z)$ is an isomorphism for all $p\leq k$ if and only if
$$F_p\gamma_Z\colon F_p\Q_Z^H[d]\to F_p\mathrm{IC}_Z\Q^H\quad \text{and}$$
$$ F_{p+d}\gamma_Z^{\vee} \colon F_{p+d}{\mathbf D}(\mathrm{IC}_Z\Q^H)=F_{p-r}W_{n+r}\cH^r_Z(\cO_X)\to
F_{p+d}{\mathbf D}\big(\Q_Z^H[d]\big)=F_{p-r}\cH_Z^r(\cO_X)$$
are isomorphisms for all $p\leq k+n$ (recall that every morphism of mixed Hodge modules preserves the Hodge filtration and it is strict). 
\end{rmk}

\begin{rmk}
We note that in \cite{FL2} one says that a variety $Z$ of pure dimension $d$ has $k$-rational singularities if the composition
$$\Omega_Z^p\to \underline{\Omega}_Z^p\overset{\psi_k}\longrightarrow\R\cH om_{\cO_Z}\big(\underline{\Omega}_Z^{d-k},\omega_Z^{\bullet}[-d]\big)$$
is an isomorphism for all $p\leq k$. It is shown in \cite[Corollary~3.17]{FL2} that this definition is equivalent to the definition we use in this paper if ${\rm codim}_Z(Z_{\rm sing})\geq 2k+1$. 
Furthermore, it is shown in \cite[Theorem~3.20]{FL2} that with their definition as well, if $Z$ is a local complete intersection and has $k$-rational singularities,
then $Z$ has Du Bois singularities, and thus ${\rm codim}_Z(Z_{\rm sing})\geq 2k+1$ by Remark~\ref{sing_locus_Du_Bois}. We thus conclude
that for local complete intersection varieties, the two definitions of $k$-rational singularities agree.
\end{rmk}

\section{Characterizations of $k$-rationality for local complete intersections}\label{section3}

Let $X$ be a smooth, irreducible variety of dimension $n$ and $Z$ be a local complete intersection subvariety of pure codimension $r$ in $X$. Let $d=n-r$ be the dimension of $Z$ and $i\colon Z\hookrightarrow X$ the inclusion. We will freely use the notation introduced in the previous section.
The following is the main result of this section, which implies several of the statements in the introduction.

\begin{thm}\label{equivalent_statements}
With the above notation, for every nonnegative integer $k$, the following conditions are equivalent: 
\begin{enumerate}
	\item $\widetilde \alpha(Z)>k+r$;
	\item $F_kW_{n+r}\cH^r_Z(\cO_X)=E_k\cH^r_Z(\cO_X)$;
	\item The morphism
	\begin{equation}\label{eqn1_equivalent_statements}
	F_{p+r}\Q_Z^H[d]\twoheadrightarrow F_{p+r}\mathrm{IC}_Z\Q^H,
	\end{equation}
	induced by $\gamma_Z$ and the composition
	\begin{equation}\label{eqn2_equivalent_statements}
	F_pW_{n+r}\cH^r_Z(\cO_X)\hookrightarrow F_p\cH^r_Z(\cO_X)\hookrightarrow E_p\cH^r_Z(\cO_X),
	\end{equation}
	induced by $\gamma_Z^{\vee}$, are isomorphisms for $p\leq k$. 
		\item $Z$ has $k$-Du Bois singularities and the morphism 
	\[
	 \psi_k\colon \underline{\Omega}_{Z}^k\to \R\cH om_{\cO_{Z}}\big(\underline{\Omega}_{Z}^{d-k},\omega_{Z}\big)
	\]
	is an isomorphism; 
	\item the canonical morphism
	\[
	\Omega^p_Z \to \gr^F_{-p}{\rm DR}_X(\mathrm{IC}_Z\Q^H)[p-d]
	\] 
	is an isomorphism (in the derived category) for $p \leq k$.
\end{enumerate}
\end{thm}

\begin{rmk}\label{rmk_van_lci}
We note that the morphism in (e) is the composition
$$\Omega_Z^p\to \underline{\Omega}_Z^p\simeq {\rm Gr}^F_{-p}{\rm DR}_X(\Q_Z^H)[p]\to 
{\rm Gr}^F_{-p}{\rm DR}_X(\mathrm{IC}_Z\Q^H)[p-d],$$
where the second map  is induced by $\gamma_Z\colon \Q_Z^H[d]\to \mathrm{IC}_Z\Q^H$. 
\end{rmk}

\begin{rmk}
Note that condition
 (d) in the theorem is equivalent to $Z$ having $k$-rational singularities 
 by Theorem~\ref{chr_k_rat}. Therefore the equivalence (a)$\Leftrightarrow$(d) is the content of Theorem~\ref{thm1_intro},
 while the equivalence (a)$\Leftrightarrow$(b) is the content of Theorem~\ref{thm2_intro}. 
 \end{rmk}
 
 We proceed with the proof of Theorem~\ref{equivalent_statements} in several steps, by showing the following implications:
 
  $(a)\Rightarrow (b) \Rightarrow (c)\Rightarrow (d)+(e)$, $(d)\Rightarrow(c)$,  and $(e)\Rightarrow (b) \Rightarrow (a)$.

\begin{proof}[Proof of Theorem~\ref{equivalent_statements}]
Since all assertions are local, we may and will assume that $X$ is affine and $Z$ is defined in $X$ by $f_1,\ldots,f_r\in\cO_X(X)$. 
In particular, we will be able to consider the $V$-filtration corresponding to these functions. We denote by $I_Z$
the ideal defining $Z$ in $X$. 

\noindent  {\bf Step 1. Proof of (a)$\Rightarrow$(b)}. Let $W_\bullet$ be the monodromy filtration on $\gr^\bullet_V B_\bff$, shifted by $n$,  uniquely characterized by: 
\begin{itemize} 
	\item $(s+\alpha)\cdot W_\bullet{\rm Gr}_V^{\alpha}B_{\bff} \subseteq W_{\bullet-2}{\rm Gr}_V^{\alpha}B_{\bff}$ and 
	\item $(s+\alpha)^j\colon \gr^W_{n+j}\gr_V^{\alpha}B_{\bff} \cong \gr^W_{n-j}\gr_V^{\alpha}B_{\bff}$ is an isomorphism 
	for all $j\geq 1$. 
\end{itemize} 
Explicitly, this is given by
\begin{equation}\label{monodromy}
W_{n+i}{\rm Gr}_V^{\alpha}B_{\bff}=\sum_j{\rm ker}\big((s+\alpha)^{i+j+1}\big)\cap {\rm Im}\big((s+\alpha)^{j}\big)
\end{equation}
Consider the map $\sigma\colon \left(\gr^{r-1}_V B_\bff\right)^{\oplus r} \xrightarrow{(t_1,t_2,\dots,t_r)} \gr^{r}_V B_\bff$. 
By~\cite{CD}*{Theorem 1.2}, for every $i$, we have an isomorphism of filtered $\cD_X$-modules
\begin{equation}\label{eq:sigma}
\gr^W_{i+r}\cH^r_Z(\cO_X)\simeq \big(\gr^W_i\coker \sigma, F[-r]\big)\quad\text{for all}\quad i\in\Z.
\end{equation}
Recall that we know that $W_{i+r}\cH^r_Z(\cO_X)=0$ for $i<n$, hence
\begin{equation}\label{eqW}
W_{n+r}\cH^r_Z(\cO_X)=\gr^W_{n+r}\cH^r_Z(\cO_X)\simeq \big(\gr^W_n\coker \sigma, F[-r]\big).
\end{equation}

Since $\widetilde\alpha(Z)>k+r$, it follows from the definition of the minimal exponent that
 $F_{k+r+1}B_\bff\subseteq V^{>r-1}B_\bff$, hence 
\[
(s+r)\cdot F_{k+r}B_\bff \subseteq \sum^r_{i=1} t_i \cdot F_{k+r+1} B_\bff \subseteq V^{>r} B_\bff.
\] 
We thus have
\begin{equation}\label{eq:ab2}
F_{k+r}\gr^r_V B_\bff \subseteq W_n\gr^r_V B_\bff
\end{equation} 
because $\ker(s+r)\subseteq W_n\gr^r_VB_\bff$ by (\ref{monodromy}). 
This implies that $F_{k+r}\coker\sigma\subseteq W_n\coker \sigma$, and using (\ref{eqW})
we conclude that
$$F_kW_{n+r}\cH^r_Z(\cO_X)=F_k\cH_Z^r(\cO_X)=E_k\cH^r_Z(\cO_X),$$
where the last equality follows from Theorem~\ref{thm_local_coho}.

\noindent {\bf Step 2. Proof of (b)$\Rightarrow$(c)}. We first prove the following

\begin{lem}\label{lem:6}
The equality $F_k W_{n+r}\cH^r_Z(\cO_X)=E_k \cH^r_Z(\cO_X)$ implies 
\[
F_p W_{n+r}\cH^r_Z(\cO_X)=E_p \cH^r_Z(\cO_X) \quad \text{for all} \quad  p\leq k.
\]
\end{lem}

\begin{proof}
Recall first that $F_p\cH^r_Z(\cO_X)\subseteq E_p\cH^r_Z(\cO_X)$ for all $p$ by \cite[Proposition~3.4]{MP1},
hence $F_p W_{n+r}\cH^r_Z(\cO_X)=E_p \cH^r_Z(\cO_X)$ if and only if the inclusion ``$\supseteq$" holds. 
 Since $W_{n+r}\cH^r_Z(\cO_X)$ is a Hodge module supported on $Z$, we have
\[
I_Z \cdot F_p W_{n+r}\cH^r_Z(\cO_X) \subseteq F_{p-1} W_{n+r}\cH^r_Z(\cO_X)\quad\text{for all}\quad p\in\Z
\] 
(see \cite[Lemme~3.2.6]{Saito_MHP}).
On the other hand, it follows easily from the definition of the filtration $E_{\bullet}\cH_Z^r(\cO_X)$ that we have
\[
 I_Z\cdot E_p\cH^r_Z(\cO_X)=E_{p-1}\cH^r_Z(\cO_X)\quad\text{for all}\quad p\geq 1.
\]
The assertion in the lemma now follows by decreasing induction on $p$.
\end{proof}

We next use duality  to prove the following

\begin{lem}\label{lem:dualfw}
If $F_pW_{n+r}\cH^r_Z(\cO_X)=F_p \cH^r_Z(\cO_X)$ for some $p\in \Z$, then the surjective map
\[
F_{p+r+1}\Q_Z^H[d] \to F_{p+r+1}{\rm IC}_Z\Q^H
\]
induced by $\gamma_Z$ is an isomorphism.
\end{lem}

\begin{proof}
The equality $F_pW_{n+r}\cH^r_Z(\cO_X)=F_p\cH^r_Z(\cO_X)$ is equivalent to
$F_p \gr^W_{n+r+j}\cH^r_Z(\cO_X)=0$ for all $j\geq 1$. Since $\gr^W_{n+r+j}\cH^r_Z(\cO_X)$ underlies a polarizable pure Hodge module of weight $n+r+j$, the choice of a polarization gives an isomorphism of filtered $\cD_X$-modules:
\begin{equation}\label{eq:dgrw1}
\bD\big(\gr^W_{n+r+j}\cH^r_Z(\cO_X)\big) \overset{\simeq}\longrightarrow \left(\gr^W_{n+r+j}\cH^r_Z(\cO_X)\right) (n+r+j).
\end{equation}
On the other hand, 
\begin{equation}\label{eq:dgrw2}
\begin{aligned}
\bD\big(\gr^W_{n+r+j}\cH^r_Z(\cO_X)\big) &\cong  \gr^W_{-n-r-j} \bD\big(\cH^r_Z(\cO_X)\big) \\
& \cong \gr^W_{-n-r-j} \left(\Q_Z^H[d](n)\right) \cong \big(\gr^W_{d-j} \Q_Z^H[d]\big)(n).
\end{aligned}
\end{equation}
Combining the two equations~\eqref{eq:dgrw1} and~\eqref{eq:dgrw2} yields
\[
\left(\gr^W_{n+r+j}\cH^r_Z(\cO_X)\right) (r+j) \cong \gr^W_{d-j}\Q_Z^H[d],
\]
as filtered $\cD_X$-modules, which implies
\[
F_{p+1-j} \gr^W_{n+r+j}\cH^r_Z(\cO_X) \cong F_{p+r+1} \gr^W_{d-j}\Q^H_Z[d].
\]
We have seen that our hypothesis gives $F_{p+1-j} \gr^W_{n+r+j}\cH^r_Z(\cO_X)=0$ for $j\geq 1$, hence
 $F_{p+r+1} \gr^W_{d-j}\Q_Z^H[d]=0$ for $j\geq 1$ and thus
  $F_{p+r+1} W_{d-1}\Q_Z^H[d]=0$. This implies the conclusion of the lemma by definition of $\gamma_Z$.
\end{proof}

Returning to the proof of the implication (b)$\Rightarrow$(c), 
note that the assertion in Lemma~\ref{lem:6}
gives the fact that the morphism (\ref{eqn2_equivalent_statements}) is an isomorphism for $p\leq k$.
Similarly, by combining Lemmas~\ref{lem:6} and \ref{lem:dualfw} we conclude that the morphism 
(\ref{eqn1_equivalent_statements}) is an isomorphism for $p\leq k$ (in fact, for $p\leq k+1$). We thus have the assertion in (c).

\noindent {\bf Step 3. Proof of (c)$\Rightarrow$(d)+(e)}.  The surjectivity of the morphism in
(\ref{eqn2_equivalent_statements}) implies, in particular, that $Z$ is $k$-Du Bois by \cite[Theorem~F]{MP1}. 
Since $\psi_k=\gr^F_{-k}{\rm DR}_X(\psi_Z)[k-d]$, we see that $\psi_k$ is an isomorphism if and only if
$\gr^F_{-k}{\rm DR}_X(\psi_Z)$ is an isomorphism, which by (\ref{eq_compat_with_duality}) holds if
and only if $\gr^F_{k-d}{\rm DR}_X(\psi_Z)$ is an isomorphism (recall that ${\mathbf D}(\psi_Z)=\psi_Z(d)$, see Remark~\ref{dual_psi}). 
We thus conclude that the condition in (d) holds. 

Similarly, once we know that $Z$ is $k$-Du Bois, the condition in (e) is equivalent with ${\rm Gr}^F_{-p}{\rm DR}_X(\gamma_Z)$ being
an isomorphism for $p\leq k$, which is equivalent by (\ref{eq_compat_with_duality}) with ${\rm Gr}^F_{p-d}{\rm DR}_X(\gamma_Z^{\vee})$
being an isomorphism for all $p\leq k$. This is implied by $F_{p+r}\gamma_Z^{\vee}$ being an isomorphism for all $p\leq k$, but this is precisely 
the morphism $F_pW_{n+r}\cH^r_Z(\cO_X)\to F_p\cH^r_Z(\cO_X)$. Therefore the condition in (e) holds as well.

\noindent {\bf Step 4. Proof of (d)$\Rightarrow$(c)}. 
It follows from Theorem~\ref{chr_k_rat} that the conditions in (d) are equivalent to $Z$ having $k$-rational singularities. In particular,
since we have these conditions for $k$, we also have them for $k-1$. In particular, we know that 
$\psi_k={\rm Gr}^F_{p-d}{\rm DR}_X(\psi_Z)[p-d]$ is an isomorphism for all $p\leq k$. Using (\ref{eq_compat_with_duality}) and the fact that
${\mathbf D}(\psi_Z)=\psi_Z(d)$, we conclude that ${\rm Gr}^F_{p-d}{\rm DR}_X(\psi_Z)$ is an isomorphism for all $p\leq k$.
Lemma~\ref{general_lemma} thus implies that $F_{p+r}\psi_Z$ is an isomorphism for all $p\leq k$. As we have seen in
Remark~\ref{rmk_two_morphisms}, this implies that the morphisms (\ref{eqn1_equivalent_statements}) and (\ref{eqn2_equivalent_statements})
are isomorphisms for all $p\leq k$ (for the latter morphism, we also use the fact that $F_p\cH_Z^r(\cO_X)=E_p\cH^r_Z(\cO_X)$ for $p\leq k$,
due to the fact that $Z$ has $k$-Du Bois singularities). We thus have the condition in (c).

\noindent {\bf Step 5. Proof of (e)$\Rightarrow$(b)}. 
If $k=0$, applying $\R\cH om_{\cO_X}\big(-,\omega_X[r]\big)$ to the isomorphism in (e) gives via (\ref{eq_compat_with_duality}) an isomorphism
$${\rm Gr}^F_{-n}{\rm DR}_XW_{n+r}\cH^r_Z(\cO_X)\to {\rm Gr}^F_{-n}{\rm DR}_X\cH^r_Z(\cO_X)\to {\mathcal Ext}^r_{\cO_X}(\cO_Z,\omega_X)$$
(note that ${\mathcal Ext}^i_{\cO_X}(\cO_Z,\omega_X)=0$ for $i\neq r$ since $Z$ is a local complete intersection of pure codimension $r$).
We have
$${\rm Gr}^F_{-n}{\rm DR}_X\cH^r_Z(\cO_X)=\omega_X\otimes_{\cO_X}F_0\cH^r_Z(\cO_X)\quad\text{and}$$
$${\rm Gr}^F_{-n}{\rm DR}_XW_{n+r}\cH^r_Z(\cO_X)=\omega_X\otimes_{\cO_X}F_0W_{n+r}\cH^r_Z(\cO_X),$$
while the image of the inclusion ${\mathcal Ext}^r_{\cO_X}(\cO_Z,\omega_X)\hookrightarrow \omega_X\otimes_{\cO_X}\cH^r_Z(\cO_X)$ is
$\omega_X\otimes_{\cO_X}E_0\cH^r_Z(\cO_X)$. We thus obtain the condition in (b) in this case.

From now on we assume $k\geq 1$. Arguing by induction on $k$, we may and will assume that 
$F_pW_{n+r}\cH^r_Z(\cO_X)=E_p\cH^r_Z(\cO_X)$ for $p\leq k-1$. In particular, we know that $Z$ has $(k-1)$-Du Bois singularities, and thus 
${\rm codim}_Z(Z_{\rm sing})\geq 2k-1\geq k$ by \cite[Corollary~3.40]{MP1}. 
We only need to prove that the injection $F_{k}W_{n+r}\cH^r_Z(\cO_X) \hookrightarrow E_{k}\cH^r_Z(\cO_X)$ is indeed an isomorphism. 
Moreover, because we know the corresponding assertions for the lower pieces of the filtrations, it follows from Lemma~\ref{general_lemma}
that it is enough to show that the induced morphism 
\begin{equation}\label{eq:dtob}
\gr^F_{k-n} \mathrm{DR}_X W_{n+r}\cH^r_Z(\cO_X)\to \gr^E_{k-n} \mathrm{DR}_X \cH^r_Z(\cO_X)
\end{equation}
is an isomorphism (in the derived category). 

Applying $\R\cH om_{\cO_X}\big(-,\omega_X[r+k]\big)$ to the isomorphism in (e) implies via (\ref{eq_compat_with_duality}) that the composition
\begin{equation}\label{eq:dtob2}
{\rm Gr}^F_{k-n}{\rm DR}_XW_{n+r}\cH^r_Z(\cO_X)\to {\rm Gr}^F_{k-n}{\rm DR}_X\cH^r_Z(\cO_X)\to \R\cH om_{\cO_X}(\Omega_Z^k,\omega_X[r+k]\big)
\end{equation}
is an isomorphism. On the other hand, since ${\rm codim}_Z(Z_{\rm sing})\geq k$, it follows from 
 \cite[Section~5.2]{MP1}
that the second map in (\ref{eq:dtob2}) gets identified with the canonical morphism
$${\rm Gr}^F_{k-n}{\rm DR}_X\cH^r_Z(\cO_X)\to {\rm Gr}^E_{k-n}{\rm DR}_X\cH^r_Z(\cO_X).$$
We thus conclude that indeed (\ref{eq:dtob}) is an isomorphism.

\noindent {\bf Step 6. Proof of (b)$\Rightarrow$(a)}. 
In addition to the map $\sigma\colon \left(\gr^{r-1}_V B_\bff\right)^{\oplus r} \xrightarrow{(t_1,t_2,\dots,t_r)} \gr^{r}_V B_\bff$ 
that we used in Step 1, we also consider the map
$\delta\colon \gr^r_V B_\bff \xrightarrow{(\partial_{t_1},\partial_{t_2},\dots,\partial_{t_r})} \big(\gr^{r-1}_V B_\bff(-1)\big)^{\oplus r}$.
The key point is to show the following

\noindent {\bf Claim}. The condition in (b) implies that 
the composition of the canonical morphisms 
\begin{equation}\label{eq:ckc}
\gr^F_{k+r}\ker\delta \hookrightarrow  \gr^F_{k+r}\gr^r_V B_\bff \twoheadrightarrow \gr^F_{k+r} \coker\sigma,
\end{equation}
is an isomorphism. 

By  \cite[Theorem~1.2]{CD}, we have an isomorphism of filtered $\cD_X$-modules
\begin{equation}\label{eq:cd12}
\gr^W_{d+i}\Q_Z^H[d] \cong \gr^W_{n+i}\ker\delta \quad\text{for all}\quad i\in\Z.
\end{equation}
In particular, we have $W_n\ker\delta=\ker\delta$ (recall that, similarly, the isomorphism (\ref{eq:sigma})
implies $W_{n-1}\coker\sigma=0$).
Note that the inclusion $W_n\ker\delta \hookrightarrow  W_n\gr^r_V B_\bff$ induces a canonical filtered morphism 
\begin{equation}\label{eq:dkcan}
\gr^W_n\ker\delta  \to W_n\coker\sigma=\frac{W_n\gr^r_V B_\bff}{W_n\gr^r_V B_\bff\cap {\rm im\,}\sigma}.
\end{equation}
Indeed, the morphsim is well-defined because
\[
W_{n-1}\ker\delta \subseteq W_{n-1}\gr^r_VB_\bff=(s+r)\cdot W_{n+1}\gr^r_VB_\bff=\sum^r_{i=1} t_i\partial_{t_i}W_{n+1}\gr^r_VB_\bff \subseteq {\rm im\,}\sigma.
\]
We deduce that the canonical map $\ker \delta \to \coker \sigma$ factors as 
\[
\ker\delta \to \gr^W_n\ker\delta \to W_n\coker\sigma \to \coker\sigma.
\]
Furthermore, the canonical maps $\gr^F_{k+r}\ker\delta \to \gr^F_{k+r}\gr^W_n\ker\delta$ and $\gr^F_{k+r}W_n\coker\sigma \to \gr^F_{k+r}\coker\sigma$ are isomorphisms because of (\ref{eq:sigma}) and (\ref{eq:cd12}), together with the fact that
the canonical map
\[
\gr^F_{p+r}\Q_Z^H[d] \to\gr^F_{p+r}\gr^W_d\Q_Z^H[d]
\]
is an isomorphism for $p\leq k+1$ by Lemma~\ref{lem:dualfw}. Therefore the claim is now reduced to the
assertion that \eqref{eq:dkcan} is a filtered isomorphism. Clearly,~\eqref{eq:dkcan} is a filtered isomorphism over the complement 
$V=X\smallsetminus Z_{\rm sing}$ of the singular locus of $Z$, due to
\[
\ker\delta\vert_V =\gr^r_V B_\bff\vert_V = \coker\sigma\vert_V
\]
preserving the Hodge filtration (this follows from the fact that if $Z$ is smooth, then $V^rB_{\bff}=B_{\bff}$ by (\ref{eq_char_V_smooth}), hence $\gr_V^{r-1}B_{\bff}=0$).
Therefore \eqref{eq:dkcan} is an isomorphism of $\cD_X$-modules because its source and target decompose by~\eqref{eq:cd12} and (\ref{eq:sigma})
as direct sums of simple $\cD$-modules, corresponding to the irreducible components of $Z$. Moreover, it is even a filtered isomorphism thanks to the fact that the Hodge filtration is uniquely determined by the regular locus~\cite{Saito_MHP}*{(3.2.2.2)}. This completes the proof of the claim.

To conclude the proof, note that the condition in (b) implies that $\widetilde\alpha(Z)\geq k+r$ by 
Theorem~\ref{thm_local_coho}. Therefore we have
$$\gr^F_{k+r}\gr^r_VB_\bff=\gr^F_{k+r}B_\bff/I_Z\cdot \gr^F_{k+r}B_\bff=\gr^F_{k+r} \coker\sigma,$$
where the first equality follows from the fact that $F_{k+r}V^{>r}B_{\bff}=\sum_{i=1}^r t_i\cdot F_{k+r}V^{>r-1}B_{\bff}$
by \cite[Theorem~1.1]{CD}. 
The claim implies that the composition
\[
\gr^F_{k+r}\ker\delta\hookrightarrow \gr^F_{k+r}\gr^r_VB_\bff \xrightarrow{=} \gr^F_{k+r} \coker\sigma,
\]
is an isomorphism, hence
$\delta$ is zero on $\gr^F_{k+r}\gr^r_V B_\bff=\gr^F_{k+r} B_\bff/I_Z\cdot\gr^F_{k+r} B_\bff$. Therefore
$$\partial_{t_i}\cdot F_{k+r}B_{\bff}\subseteq F_{k+r}B_{\bff}+V^{>r-1}B_{\bff}\subseteq V^rB_{\bff}+V^{>r-1}B_{\bff}\subseteq V^{>r-1}B_{\bff}\quad \text{for}\quad 1\leq i\leq r,$$
where the second inclusion comes from the fact that $\widetilde{\alpha}(Z)\geq k+r$. 
This implies $F_{k+r+1}B_{\bff}\subseteq V^{>r-1}B_{\bff}$, which is equivalent to $\widetilde{\alpha}(Z)>k+r$. 
This completes the proof of this step and thus the proof of the theorem.
\end{proof}

We next prove the two corollaries stated in the Introduction:

\begin{proof}[Proof of Corollary~\ref{cor2_intro}]
The assertion follows from the fact that $Z$ has $k$-Du Bois singularities if and only if $\widetilde{\alpha}(Z)\geq k+r$,
while by Theorem~\ref{thm1_intro}, $Z$ has $(k-1)$-rational singularities if and only if $\widetilde{\alpha}(Z)>k+r-1$. 
\end{proof}

\begin{proof}[Proof of Corollary~\ref{cor2.5_intro}]
We may assume that $Z$ is irreducible and affine and let $Z\hookrightarrow X$ be a closed embedding, of codimension $r$, with $X$ a smooth,
irreducible variety. The assertion to prove is trivial if $Z$ is smooth (with the convention that the empty set has infinite codimension),
hence we may and will assume that $Z$ is singular. If $s=\dim(Z_{\rm sing})$ and
$H$ is the intersection of general hyperplanes sections in $X$, then $Z':=Z\cap H$ is a local complete intersection variety
with nonempty, $0$-dimensional singular locus,
and $\widetilde{\alpha}(Z')=\widetilde{\alpha}(Z)$ by \cite{CDMO}*{Theorem 1.2}. In particular, it follows from
Theorem~\ref{thm1_intro} that $\widetilde{\alpha}(Z')>k+r$. 
Since ${\rm codim}_Z(Z_{\rm sing})={\rm codim}_{Z'}(Z'_{\rm sing})$, we may
replace $Z$ by $Z'$ to assume that
$Z_{\rm sing}$ is nonempty and zero-dimensional. We then need to show that $d:=\dim(Z)\geq 2k+2$. 

Let $x\in Z_{\rm sing}$. By (\ref{eq_bd_min_exp}), we have
\[
\widetilde{\alpha}(Z) \leq \dim(X)-\frac{1}{2}\dim_{\C}T_x(Z)=(d+r)-\frac{1}{2}\dim_{\C}T_x(Z).
\] 
Since $x\in Z_{\rm sing}$, we have $\dim_{\C}T_x(Z) \geq \dim(Z)+1=d+1$, hence
\[
\widetilde \alpha(Z) \leq (d+r)-\frac{1}{2}(d+1)=\frac{d-1}{2}+r.
\]
Since $\widetilde\alpha(Z)> k+r$, we conclude that $k+r<\frac{d-1}{2}+r$, hence $d>2k+1$.
We thus conclude that $d\geq 2k+2$. 
\end{proof}

\section{Non-vanishing result for the Du Bois complex}\label{section_nonvanishing}

In this section we show that for singular, $d$-dimensional, local complete intersection varieties $Z$ with $k$-rational singularities, where $k\geq 1$, 
the cohomology sheaf $\cH^k(\underline{\Omega}_Z^{d-k})$ does not vanish.

\begin{proof}[Proof of Theorem~\ref{thm_nonvanishing}]
Note first that Theorem~\ref{chr_k_rat} gives an isomorphism 
$$\Omega_Z^k\simeq \R\cH om_{\cO_Z}(\underline{\Omega}_Z^{d-k},\omega_Z),$$
and since $\R\cH om_{\cO_Z}(-,\omega_Z)$ is a duality, we get an isomorphism
$$\underline{\Omega}_Z^{d-k}\simeq \R\cH om_{\cO_Z}(\Omega_Z^k,\omega_Z).$$
The first isomorphism in the theorem follows by taking that $k$-th cohomology sheaf. 

It is shown in \cite[Section~5.2]{MP1} that since
$Z$ has $k$-Du Bois singularities (more precisely, since ${\rm codim}_Z(Z_{\rm sing})\geq k$), the sheaf $\Omega_Z^k$ is the $0$-th
cohomology of the complex
$$0\to {\rm Sym}^k({\mathcal N}_{Z/X})^{\vee}\to \Omega_X^1\otimes_{\cO_X}{\rm Sym}^{k-1}({\mathcal N}_{Z/X})^{\vee}\to
\ldots\to \Omega_X^{k-1}\otimes_{\cO_X}{\mathcal N}_{Z/X}^{\vee}\to\Omega_X^k\otimes_{\cO_X}\cO_Z\to 0,$$
placed in cohomological degrees $-k,\ldots,0$. Since this is a resolution of $\Omega_Z^k$ by locally free $\cO_Z$-modules, it follows that
$${\mathcal Ext}^k_{\cO_Z}(\Omega_Z^k,\omega_Z)\simeq \omega_Z\otimes_{\cO_Z}{\mathcal Ext}^k_{\cO_Z}(\Omega_Z^k,\cO_Z)$$
$$\simeq\omega_Z\otimes_{\cO_Z} {\rm coker}\big({\mathcal T}_X\otimes_{\cO_X}{\rm Sym}_{\cO_Z}^{k-1}({\mathcal N}_{Z/X})\to 
{\rm Sym}_{\cO_Z}^{k}({\mathcal N}_{Z/X})\big)\simeq \omega_Z\otimes_{\cO_Z}{\rm Sym}^k_{\cO_Z}({\mathcal Q}),$$
where the last isomorphism follows from \cite{DE}*{Proposition A2.2(d)}.

In order to see that $\cH^k(\underline{\Omega}_Z^{d-k})_x\neq 0$ if $x\in Z$ is a singular point, it is enough to consider, in a neighborhood of $x$,
a closed immersion $Z\hookrightarrow X$ such that $T_xZ=T_xX$. In this case the morphism of locally free $\cO_Z$-modules
$${\mathcal T}_X\otimes_{\cO_X}{\rm Sym}_{\cO_Z}^{k-1}({\mathcal N}_{Z/X})\to 
{\rm Sym}_{\cO_Z}^{k}({\mathcal N}_{Z/X})$$
is given by a matrix whose entries all vanish at $x$. We thus conclude that the minimal number of generators of
$\cH^k(\underline{\Omega}_Z^{d-k})_x$ is equal to ${\rm rank}\big({\rm Sym}_{\cO_Z}^k({\mathcal N}_{Z/X})\big)={{e-d+k-1}\choose k}$,
where $e=\dim_{\C}T_xZ$, hence it is nonzero since $e\geq d+1$. This concludes the proof.
\end{proof}

\section{Generation level of the Hodge filtration in terms of the minimal exponent}\label{section4}

In this section we prove the bound on the level of generation in terms of the minimal exponent.

\begin{proof}[Proof of Theorem~\ref{thm4_intro}]
Let $d=\dim(Z)=n-r$. 
The starting point is the observation that for every mixed Hodge module $M$ on $X$, the Hodge filtration is generated at level $q$ if $\cH^0 \gr^F_{p-n} \DR_X(M) = 0$ for all $p >q$. Recall that by (\ref{eq_compat_with_duality}), we have 
\[  \gr^F_{p-n} \DR_X(M)\simeq \R\cH om_{\cO_X}\big(\gr^F_{n-p}\DR_X(\mathbf D(M)),\omega_X[n]\big).\]
If we apply this with $M=\cH^ri^!\Q_X^H[n]=\cH^r_Z(\cO_X)$, where $i\colon Z\hookrightarrow X$ is the inclusion,
since ${\mathbf D}(M)\simeq \Q^H_Z[d](n)$, we conclude that 
\[ \gr^F_{p-n} \DR_X \cH^r_Z(\cO_X) \simeq \R\cH om_{\cO_X}\big(\gr^F_{-p}\DR_X\Q_Z^H[d] ,\omega_X[n]\big).\]
If we apply $\cH^0(-)$ on both sides, we conclude that the Hodge filtration on $\cH^r_Z(\cO_X)$ is generated at level $q$ if 
\begin{equation}\label{eq1_thm4}
\cE xt^n_{\cO_X}\big(\gr^F_{-p} \DR_X\Q_Z^H[d],\cO_X\big) = 0
\end{equation}
for all $p>q$.

Recall now that for a bounded complex of $\cO_X$-modules $K^\bullet$ and an $\cO_X$-module $\cF$, 
there is a spectral sequence 
\[ E_1^{i,j} = \cE xt_{\cO_X}^j(K^{-i},\cF) \implies \cE xt^{i+j}_{\cO_X}(K^\bullet,\cF).\]
We take $K^{\bullet}$ to be the complex ${\rm Gr}^F_{-p}{\rm DR}_X\Q_Z^H[d]$, so that
$$K^{-\ell}=\Omega_X^{n-\ell}\otimes_{\cO_X}\gr^F_{n-\ell-p}\Q_Z^H[d].$$
Therefore the vanishing in (\ref{eq1_thm4}) holds if 
for all $j\in\{0,\ldots,n\}$, we have 
\[ \cE xt_{\cO_X}^j \big(\Omega_X^{n-(n-j)}\otimes_{\cO_X}\gr^F_{n-(n-j)-p}\Q_Z^H[d],\omega_X\big)=0,\]
or equivalently,
\[ \cE xt_{\cO_X}^j \big(\gr^F_{j-p}\Q_Z^H[d], \cO_X\big) = 0.\]
We conclude that in order to complete the proof of the theorem, it is enough to show 
the following claim:
\begin{claim} For all $p \geq n - \lceil \widetilde{\alpha}(Z) \rceil$ and all $j\in \{0,1,\dots, n\}$, we have
\begin{equation}\label{eq_claim}
\cE xt_{\cO_X}^j \big(\gr^F_{j-p}\Q_Z^H[d], \cO_X\big) = 0.
\end{equation}
\end{claim}

In order to prove the claim, we may and will assume that $X$ is affine and $Z$ is defined by $f_1,\ldots,f_r\in\cO_X(X)$, so we can make use
of the corresponding $V$-filtration. By Theorem~\ref{thm_CD}, for every $\ell\in\Z$, we have an isomorphism
$$F_{\ell}\Q^H_Z[d]\simeq \ker\big(F_{\ell}{\rm Gr}_V^r(B_{\bff})\overset{\partial_{t_1},\ldots,\partial_{t_r}}\longrightarrow
\bigoplus_{i=1}^rF_{\ell}{\rm Gr}_V^{r-1}(B_{\bff})(-1)\big).$$
Suppose now that $\widetilde{\alpha}(Z)>\ell$. In this case, by definition of the minimal exponent we have
$F_{\ell+1}B_{\bff}\subseteq V^{>r-1}B_{\bff}$ and $F_{\ell}B_{\bff}
\subseteq V^rB_{\bff}$. We thus conclude that 
\begin{equation}\label{eq2_gen_level}
F_{\ell}\Q^H_Z[d]\simeq F_{\ell}{\rm Gr}_V^r(B_{\bff})\simeq F_{\ell}B_{\bff}/F_{\ell}V^{>r}B_{\bff}.
\end{equation}
On the other hand, it follows from \cite[Theorem~1.1]{CD} that we have
$$F_{\ell}V^{>r}B_{\bff}=\sum_{i=1}^rt_i\cdot F_{\ell}V^{>r-1}B_{\bff}=\sum_{i=1}^rt_i\cdot F_{\ell}B_{\bff},$$
so that (\ref{eq2_gen_level}) gives
$$F_{\ell}\Q^H_Z[d]\simeq F_{\ell}B_{\bff}/(t_1,\ldots,t_r)F_{\ell}B_{\bff},$$
and thus
$${\rm Gr}_{\ell}^F\Q_Z^H[d]\simeq {\rm Gr}^F_{\ell}B_{\bff}/(t_1,\ldots,t_r){\rm Gr}^F_{\ell}B_{\bff}.$$
Recall now that by definition we have 
 $F_\ell B_\bff = \gr^F_\ell B_\bff = 0$ if $\ell < r$ and ${\rm gr}^F_{\ell}B_{\bff}=\bigoplus_{|\beta| = \ell -r}\cO_X\partial_t^{\beta}$ if $\ell\geq r$,
 with each  $t_i$ acting as multiplication by $f_i$. We thus conclude that if $\ell\geq r$, then
 \begin{equation}\label{eq3_gen_level}
 {\rm Gr}_{\ell}^F\Q_Z^H[d]\simeq\bigoplus_{|\beta| = \ell -r}\cO_Z\partial_t^{\beta}.
 \end{equation}

We now proceed to prove the claim. Note that since $Z$ is singular, it follows from (\ref{eq_bd_min_exp}) that 
$\widetilde{\alpha}(Z)\leq n-\tfrac{1}{2}(d+1)$, hence $n-\lceil \widetilde{\alpha}(Z)\rceil\geq\lfloor (d+1)/2)\rfloor\geq 1$.

We first consider the case when $p > n- \lceil \widetilde{\alpha}(Z) \rceil$, so that $p>0$ and $\lceil \widetilde{\alpha}(Z) \rceil > n -p \geq j-p$ for all $j \in \{0,1,\dots, n\}$. By taking $\ell=j-p$, it follows from (\ref{eq3_gen_level}) that we have
\[ \gr^F_{j-p}\Q^H_Z[d] \simeq \begin{cases} 0& \text{if}\,\,j-p < r \\ \bigoplus_{|\beta| = j-p -r} \cO_Z \de_t^\beta& \text{if}\,\,j-p \geq r\end{cases}.\]
Clearly, the vanishing in (\ref{eq_claim}) holds if $j-p<r$.  If $j\geq r+p > r$, we use the fact that $Z$ is a complete intersection, so
locally we have the Koszul resolution of $\cO_Z$, of length $r$, by free $\cO_X$-modules. 
In particular, we have $\cE xt^j_{\cO_X}(\cO_Z,\cO_X) = 0$ for all $j>r$, proving the claim in this case. 

We next consider the case when $p = n -\lceil \widetilde{\alpha}(Z) \rceil$. If $j \in \{0,1,\dots, n-1\}$, then $\lceil \widetilde{\alpha}(Z) \rceil>j-p$ and 
we get the vanishing in (\ref{eq_claim}) as above. In order to complete the proof of the claim, it is thus enough to consider $j=n$ and show that
\begin{equation}
\cE xt^n_{\cO_X}\big(\gr^F_{\lceil \widetilde{\alpha}(Z) \rceil} \Q_Z^H[d],\cO_X\big) = 0.
\end{equation}

It follows from Theorem~\ref{thm_CD} that we have an inclusion 
$\gr^F_{\lceil \widetilde{\alpha}(Z) \rceil} \Q_Z^H[d] \subseteq \gr^F_{\lceil \widetilde{\alpha}(Z) \rceil} \gr^r_VB_\bff$.
Since ${\mathcal Ext}_{\cO_X}^{n+1}(-,\cO_X)=0$, we deduce using the long exact sequence of 
$\cE xt$ sheaves that we have a surjection
\[\cE xt^n_{\cO_X}\big(\gr^F_{\lceil \widetilde{\alpha}(Z) \rceil} \gr^r_V B_\bff,\cO_X\big) \to \cE xt^n_{\cO_X}
\big(\gr^F_{\lceil \widetilde{\alpha}(Z) \rceil}\Q_Z^H[d],\cO_X\big).\]
Therefore it is enough to show that the left term is 0.

Note now that it follows from \cite[Theorem~1.1]{CD} that
$$F_{\lceil \widetilde{\alpha}(Z) \rceil} V^{>r}B_\bff = (t_1,\ldots,t_r)F_{\lceil \widetilde{\alpha}(Z)\rceil}V^{>r-1}B_{\bff}=
(t_1,\dots, t_r)F_{\lceil \widetilde{\alpha}(Z) \rceil} B_\bff,$$
where the second equality follows from the definition of the minimal exponent.
Therefore we have
\[\begin{aligned} \gr^F_{\lceil \widetilde{\alpha}(Z) \rceil} \gr^r_V B_\bff&=\gr^F_{\lceil \widetilde{\alpha}(Z) \rceil}V^rB_{\bff}/(t_1,\ldots,t_r){\rm Gr}^F_{\lceil \widetilde{\alpha}(Z) \rceil}B_{\bff} \\
&\subseteq  \gr^F_{\lceil \widetilde{\alpha}(Z) \rceil}B_\bff/(t_1,\dots, t_r) \gr^F_{\lceil \widetilde{\alpha}(Z) \rceil}B_\bff. \end{aligned}\]
Using again the fact that ${\mathcal Ext}^{n+1}_{\cO_X}(-,\cO_X)=0$, we see that it is enough to show that
\[ \cE xt^n_{\cO_X}(\gr^F_{\lceil \widetilde{\alpha}(Z) \rceil}B_\bff/(t_1,\dots, t_r) \gr^F_{\lceil \widetilde{\alpha}(Z) \rceil}B_\bff,\cO_X) =0.\]
This follows from the fact that $\gr^F_{\lceil \widetilde{\alpha}(Z) \rceil}B_\bff/(t_1,\dots, t_r) \gr^F_{\lceil \widetilde{\alpha}(Z) \rceil}B_\bff$
is isomorphic to a direct sum of copies of $\cO_Z$ and ${\mathcal Ext}^n_{\cO_X}(\cO_Z,\cO_X)=0$, as follows using the 
Koszul resolution of $\cO_Z$ (note that $r<n$, since we assume that $Z$ is reduced and singular). This completes the proof of the claim and thus
the proof of the theorem.
\end{proof}

\end{document}